\DeclareMathAlphabet{\mathpzc}{OT1}{pzc}{m}{it}

\documentclass[headsepline=true]{scrartcl}
\usepackage{amsmath}
\usepackage{amsthm, 	amssymb, amsfonts}
\usepackage[top=1.2in,bottom=1.2in,left=1in,right=1in]{geometry}
\usepackage[pdfborder={0 0 0}]{hyperref}
\usepackage{amscd}
\usepackage{tensor}
\usepackage{mathrsfs}
\usepackage{arydshln}
\usepackage{tikz}
\usepackage[utf8]{inputenc}

\usetikzlibrary{decorations.markings}
\usetikzlibrary{decorations.pathreplacing}
\usepackage{lscape}
\usepackage{enumerate}
\usepackage[capitalize]{cleveref}

\usepackage{caption}
\usepackage{mathtools}

\DeclarePairedDelimiter\floor{\lfloor}{\rfloor}

\renewcommand{\phi}{\varphi}
\renewcommand{\d}{\mathrm{d}}

\theoremstyle{plain} 
\newtheorem{thm}{Theorem}[section]
\newtheorem{lem}[thm]{Lemma} 
\newtheorem{cor}[thm]{Corollary} 
\newtheorem{prop}[thm]{Proposition}

\theoremstyle{definition} 

\newtheorem{exmp}[thm]{Example}

\theoremstyle{remark}
\newtheorem{remarkx}{Remark}
\newenvironment{remark}
{\pushQED{\qed}\remarkx}
{\popQED\endremarkx}

\title{Special zeta Mahler functions}
\author{Berend Ringeling\thanks{This work is supported by NWO grant OCENW.KLEIN.006.}
	\\[1mm]
	\small Department of Mathematics, IMAPP, Radboud University,\\[-1mm] \small PO Box 9010, 6500~GL Nijmegen, Netherlands\\[0.5mm] \small \url{b.ringeling@math.ru.nl}
}
\date{\today}

\begin{document}
	
	\maketitle
	
	\begin{abstract}
		In 1969, I. Bernstein and S. Gelfand introduced an object, which is now called the \emph{zeta Mahler function} (ZMF, also \emph{zeta Mahler measure}) and related to the \emph{Mahler measure}. 
		Here we discuss a family of ZMFs attached to the Laurent polynomials $k + (x_1 + x_1^{-1}) \cdots (x_r + x_r^{-1})$, where $k$ is real. We give explicit formulae, present examples and establish properties for these ZMFs, such as an RH-type phenomenon. Further, we explore relations with the Mahler measure.
		
		
	\end{abstract}
	
	\section{Introduction}
	
	For a Laurent polynomial $P \in \mathbb{C}[x_1^{\pm 1}, \dots , x_r^{\pm 1}] \setminus \{ 0 \},$ define the \emph{zeta Mahler function} (ZMF) as
	\begin{equation}
		\label{EQNN}
		Z(P;s) := \frac{1}{(2 \pi i)^r}\int_{\mathbb{T}^r} \left| P(x_1, \dots, x_r) \right|^s \, \frac{\d x_1}{x_1} \dots \frac{\d x_r}{x_r},
	\end{equation}
	where $s$ is a complex parameter.
	This function was first introduced by S. Gelfand and I. Bernstein in \cite{Bernstein}. The function, though still unnamed, was later studied by M. Atiyah \cite{Atiyah},  I. Bernstein \cite{Bernstein2} and by  Cassaigne and Maillot \cite{Cassaigne}. Finally, in 2009 these functions were re-introduced by Akatsuka \cite{AKATSUKA}.
	The values $Z(P;s)$ can be interpreted as the average value of $\left| P(x_1, \dots, x_r) \right|^s$ on the torus $\mathbb{T}^r = \{ (x_1, \dots, x_r) \in \mathbb{C}^r \colon |x_1| = \dots = |x_r| = 1\}$. If we let $X_1, \dots, X_r$ be uniformly distributed random variables on the complex unit circle we can also interpret $Z(P;s)$ as the $s$-th moment of the random variable $|P(X_1, \dots , X_r)|$.
	The \emph{logarithmic Mahler measure} of $P$ is defined as
	\begin{equation}
		\operatorname{m}(P) := \frac{1}{(2 \pi i)^r}\int_{\mathbb{T}^r} \log|P(x_1, \dots , x_r)| \, \frac{\d x_1}{x_1} \dots \frac{\d x_r}{x_r} = \frac{\d Z(P;s)}{\d s}\Bigr|_{s = 0},
	\end{equation}
	see \cite{Mmbook}.
	
	In this paper we discuss properties of the zeta Mahler function of the polynomials $k + (x_1 + x_1^{-1}) \cdots (x_r + x_r^{-1})$ for real $k$. We denote this function by $W_r(k;s)$,
	\begin{equation}
		\label{TheIntegral}
		W_r(k;s) = \frac{1}{(2 \pi i)^r} \int_{\mathbb{T}^r} \left|k + (x_1 + x_1^{-1}) \cdots (x_r + x_r^{-1}) \right|^s  \frac{\d x_1}{x_1} \dots \frac{\d x_r}{x_r}.
	\end{equation}
	It is easy to see that for real $k$, the quantity $k + (x_1 + x_1^{-1}) \cdots (x_r + x_r^{-1})$ is real-valued on the torus $\mathbb{T}^r$. Further, a substitution $x_1 \mapsto -x_1$ shows that $W_r(|k|,s) = W_r(k;s)$, so it suffices to consider only $k \geq 0$. In the computation we make a clear distinction between the cases $k \geq 2^r$ and $0 \leq k < 2^r$. In the case $k \geq 2^r$ (the ``light'' case) the structure is much simpler as we can drop the absolute value in the integral \eqref{TheIntegral}.
	The ZMFs are originally defined in a half-plane $\operatorname{Re}(s)>s_0$ for some $s_0<0$, and the hypergeometric formulae provide one with an efficient way for analytic continuation of the ZMF to a meromorphic function of $s$ with an explicit location and structure of poles.
	
	For $r=1$, these values are recorded in \cite[Theorem 4]{AKATSUKA}. We give a simpler expression for $W_1(k;s)$.
	
	\begin{thm}
		\label{1.1}
		For $s \in \mathbb{C}$, the following expressions are valid.
		\
		\begin{enumerate}
			\item For $|k|>2$, we have
			\begin{equation}
				W_1(k;s) = |k|^s \cdot {}_{2} F_{1} \left( \frac{-s}{2}, \frac{1-s}{2} ; 1;     \frac{4}{k^2} \right).
			\end{equation}
			\item For $|k| = 2$,  we have
			\begin{equation}
				W_1(k;s) = \frac{2^s \Gamma(\frac{1}{2} + s)}{\Gamma(1+\frac{s}{2}) \Gamma(\frac{1+s}{2})}.
			\end{equation}
			\item For $|k| < 2$, we have
			\begin{equation}
				\label{1.1.iii}
				W_{1}(k;s) = \frac{4^s \Gamma(\frac{1 + s}{2})^2}{\pi \Gamma(1 + s)} \cdot {}_{2} F_{1} \left( \frac{-s}{2}, \frac{-s}{2} ; \frac{1}{2}; \frac{k^2}{4} \right).
			\end{equation}
		\end{enumerate}
	\end{thm}
	Here and in what follows
	\begin{equation*}
		{}_{r+1} F_{r} \left( a_1, \dots, a_{r+1}; b_1, \dots ,b_{r}; z\right) := \sum_{n \geq 0} \frac{(a_1)_n \cdots (a_{r+1})_n}{(b_1)_n \cdots (b_r)_n} \frac{z^n}{n!}
	\end{equation*}
	is the \emph{hypergeometric function}, where $$(x)_n := x (x + 1) \cdots (x + n - 1)$$
	denotes the \emph{Pochhammer symbol}.
	A generalization of the hypergeometric function we also need is the \emph{Meijer G-function}, 
	\begin{equation*}
		G_{p,q}^{m,n}(a_1, \dots, a_p;b_1, \dots, b_q; z); 
	\end{equation*}
	see \cite[Section 16.17]{article} for the definition.
	
	The explicit formulae in Theorem \ref{1.1} allow us to compute a functional equation in the variable $s$, using the symmetry of the ${}_{2} F_{1}$ hypergeometric function.
	\begin{thm}
		\label{1.2}
		\
		\begin{enumerate}
			\item For $|k|>2$ and $s \in \mathbb{C}$, we have
			\begin{equation}
				\label{1.2.i}
				W_1(k;-s-1) = (k^2 - 4)^{-s-\frac{1}{2}} W_1(k;s).
			\end{equation}
			\item For $|k|<2$ and $-1 < \operatorname{Re}(s) < 0$, we have
			\begin{equation}
				\label{1.2.ii}
				W_1(k;-s-1) = \cot \left(-\frac{\pi s}{2} \right) (4-k^2)^{-s-\frac{1}{2}} W_1(k;s).
			\end{equation}
		\end{enumerate}
	\end{thm}
	The functional equation for $|k|>2$ was already recorded in \cite[Theorem 3]{AKATSUKA}, the $|k|<2$ functional equation is new. Note that in both cases the line of symmetry of $W_1(k;s)$ is $\operatorname{Re}(s) = -\frac{1}{2}$. In fact, all the zeros of $W_1(k;s)$ lie on this line. 
	Furthermore, equation \eqref{1.2.ii} defines an analytic continuation of $W_1(k;s)$ to a meromorphic function on $\mathbb{C}$.
	\begin{thm}
		\label{1.3}
		For any $k \in \mathbb{R}$, all the non-trivial zeros of $W_1(k;s)$ lie on the critical line $\operatorname{Re}(s) = -\frac{1}{2}$.
	\end{thm}
	Furthermore, we present a general formula for all $r$ in the ``light'' case.
	\begin{thm}
		\label{1.4}
		Let $r \geq 1$.
		\
		\begin{itemize}
			\item[(i)] For $|k| > 2^r$ and all $s \in \mathbb{C}$,
			\begin{equation}
				\label{1.4.i}
				W_r(k;s) = |k|^s \cdot {}_{r+1} F_{r} \left( \frac{-s}{2}, \frac{1-s}{2}, \frac{1}{2}, \dots , \frac{1}{2}; 1, \dots, 1; 
				\frac{4^r}{k^2} \right).
			\end{equation}
			\item[(ii)] For $|k| = 2^r$ and $\operatorname{Re}(s) > -\frac{r}{2}$,
			\begin{equation}
				\label{1.4.ii}
				W_r(k;s) = 2^{rs} {}_{r+1} F_{r} \left( \frac{-s}{2}, \frac{1-s}{2}, \frac{1}{2}, \dots , \frac{1}{2}; 1, \dots, 1; 
				1 \right).
			\end{equation}
		\end{itemize}
	\end{thm}
	The case $r=2$ and $|k|>4$ was already recorded in \cite[Theorem 6]{AKATSUKA}.
	
	For $0 < k < 2^r$ we show that, as a function of $k$, $W_r(k;s)$ always satisfies the same differential equation as $W_r(k;s)$ for $|k|>2^r$. So that the knowledge of the solutions of this differential equation allows us to give an explicit formula for $W_r(k;s)$ for $|k| < 2^r$. 
	For real $s$ we have the following result.
	\begin{thm}
		\label{1.5}
		For real $s>0$, $s$ not an odd integer, and real $k$,
		\begin{align}
			W_r(k;s) &= |k|^s \operatorname{Re} \, {}_{r+1} F_{r} \left( \frac{-s}{2}, \frac{1-s}{2}, \frac{1}{2}, \dots , \frac{1}{2}; 1, \dots , 1; 
			\frac{4^r}{k^2} \right) \nonumber \\ &\qquad + \tan\left(\frac{\pi s}{2} \right) |k|^s \operatorname{Im} \, {}_{r+1} F_{r} \left( \frac{-s}{2}, \frac{1-s}{2}, \frac{1}{2}, \dots , \frac{1}{2}; 1, \dots , 1; 
			\frac{4^r}{k^2} \right).
		\end{align}
	\end{thm}
	
	More generally, for $r=2$ we find the following formula.
	\begin{thm}
		\label{1.6}
		For $|k| < 4$ and $\operatorname{Re}(s) > -1$ and $s$ not an odd integer, we have
		\begin{align}
			W_{2}(k;s) =\frac{1}{2 \pi} \frac{\tan(\frac{\pi s}{2})}{s+1} |k|^{1+s} {}_3 F_{2} \left(\frac{1}{2},\frac{1}{2} ,\frac{1}{2};1 +\frac{s}{2}, \frac{3}{2} + \frac{s}{2} ; \frac{k^2}{16} \right) \nonumber \\ + \frac{\Gamma(s+1)^2}{\Gamma(\frac{s}{2}+1)^4}  {}_3 F_{2}  \left(\frac{-s}{2} ,\frac{-s}{2},\frac{-s}{2}; \frac{1-s}{2},\frac{1}{2}; \frac{k^2}{16} \right).    
		\end{align}
	\end{thm}
	Furthermore for $r=3$ we find the following extension.
	\begin{thm}
		\label{1.7}
		For $|k|<8$ and $\operatorname{Re}(s) > -1$, $s$ not an odd positive integer, we have
		\begin{align}
			\label{W3}
			W_3(k;s) = &\frac{\Gamma(1+s)^3}{\Gamma(1+\frac{s}{2})^6} \cdot {}_4 F_{3} \left(\frac{-s}{2},\frac{-s}{2},\frac{-s}{2},\frac{-s}{2};\frac{1-s}{2},\frac{1-s}{2},\frac{1}{2};\frac{k^2}{64}\right) \nonumber\\
			&-\frac{\tan(\frac{\pi s}{2})^2}{4 \pi (1+s)} |k|^{1+s} \cdot {}_4 F_{3} \left(\frac{1}{2},\frac{1}{2},\frac{1}{2},\frac{1}{2};1,1+\frac{s}{2},\frac{3+s}{2};\frac{k^2}{64}\right) \nonumber\\
			& +\frac{4^s \tan(\frac{\pi s}{2}) \Gamma(s+1)}{\pi^{7/2}} \cdot G_{4,4}^{2,4} \left( \frac{2+s}{2}, \frac{2+s}{2}, \frac{2+s}{2}, \frac{2+s}{2};\frac{1+s}{2}, \frac{1+s}{2}, 0 ,\frac{1}{2} ;\frac{k^2}{64} \right),
		\end{align}
		where $G_{p,q}^{m,n}$ denotes the Meijer G-function.
	\end{thm}
	Finally, define $p_r(k;-)$ to be the probability density function of the random variable $|k + (X_1 + X_1^{-1}) \cdots (X_r + X_r^{-1})|$, where the $X_i$ are independent uniformly distributed random variables on the complex unit circle. 
	We can relate this $p_r$ to $W_r$ via
	\begin{equation}
		W_r(k;s) = \int_0^\infty x^s p_r(k;x)\, \d x.
	\end{equation}
	So that $p_r(k;-)$ is the inverse Mellin transform of $W_r(k;s)$ (see Section 4.1.2).
	
	This paper is structured as follows.
	In Section 2 we prove Theorem \ref{1.1}, and in Section 3 we prove Theorems \ref{1.2} and \ref{1.3}. In Section 4, we study properties of the probability densities $p_r$ and prove Theorems \ref{1.4} and \ref{1.5}, using methods different from Section 2 and generalizing Theorem \ref{1.1}. We finish Section 4 with the proofs of Theorems \ref{1.6} and \ref{1.7}.
	

	
	
	
	
	\section{Proof of Theorem \ref{1.1}}
	\subsection{Proof of Theorems \ref{1.1}.(i) and \ref{1.1}.(ii)}
	Assume that $k>2$. We compute the probability distribution $p_1(k;-)$ of the random variable $|k + X + X^{-1}|$, where $X$ is uniformly distributed on the unit circle $\{z \in \mathbb{C} \colon |z| = 1 \}$. It can be easily seen that the density of this random variable for $k>2$ is given by 
	\begin{equation}
		\label{p1}
		p_1(k;x) = \frac{1}{2\pi} \frac{1}{\sqrt{1  - \frac{(x - k)^2}{4}}}
	\end{equation}
	with the support on $(k-2 , k + 2)$, so that
	
	\begin{align*}
		W_1(k;s) = \int_{k-2}^{k+2} x^s  p_1(k;x) \d x &= \frac{1}{2 \pi} \int_{k-2}^{k+2} \frac{x^s \d x}{\sqrt{1 - \frac{(x-k)^2}{4}}}.
	\end{align*}
	Using the substitution $y = (x-k+2)/4$ to normalize the integral, we find
	\begin{equation*}
		W_1(k;s) = \frac{(k-2)^s}{\pi} \int_0^1 y^{-\frac{1}{2}} (1-y)^{-\frac{1}{2}} \left(1 - \frac{4}{2 - k} y \right)^s \, \d y.
	\end{equation*}
	Recall now the Euler-type integral
	\begin{equation}
		\label{EulerType}
		{}_2 F_{1}  \left( a,b;c;z \right) = \frac{\Gamma(c)}{\Gamma(b) \Gamma(c-b)} \int_0^1 t^{b-1} (1-t)^{c-b-1} (1 - zt)^{-a} \, \d t,
	\end{equation}
	valid for $\operatorname{Re}(c) > \operatorname{Re}(b) > 0$ and $z \not \in [1, \infty)$, see \cite[p. 4]{bailey1935generalized}.
	We may write
	\begin{equation}
		\label{DEEZ}
		W_1(k;s) = (k-2)^s \cdot {}_{2} F_{1} \left( -s, \frac{1}{2} ; 1;     \frac{4}{2 - k} \right).
	\end{equation}
	Applying the quadratic transformation
	\begin{equation*}
		{}_2 F_{1}  \left( a,b;2b;z \right) = \left(1 - \frac{1}{2}z\right)^{-a}     {}_2 F_{1}  \left( \frac{1}{2}a,\frac{1}{2}a+\frac{1}{2};b+\frac{1}{2};\left(\frac{z}{2-z}\right)^2 \right)
	\end{equation*}
	for $|\operatorname{arg}(1-z)|<\pi$ (see \cite[Eqn. (9.6.17)]{Lebedev}), equation
	\eqref{DEEZ} simplifies to
	\begin{equation}
		\label{DEEZ2}
		W_1(k;s) = k^s \cdot {}_{2} F_{1} \left( \frac{-s}{2}, \frac{1-s}{2} ; 1;     \frac{4}{k^2} \right).
	\end{equation}
	For $k = \pm 2$, we find
	\begin{equation}
		W_1(\pm 2;s) = 2^s \cdot {}_{2} F_{1} \left( \frac{-s}{2}, \frac{1-s}{2} ; 1; 1 \right)
	\end{equation}
	valid for $\operatorname{Re}(s) > -\frac{1}{2}$.
	Letting $z \to 1$ in \eqref{EulerType} (in other words, using the Gauss summation formula) gives
	\begin{equation}
		W_1(\pm 2;s) = \frac{2^s \Gamma(\frac{1}{2} + s)}{\Gamma(1+\frac{s}{2}) \Gamma(\frac{1+s}{2})}.
	\end{equation}
	\subsection{Proof of Theorem \ref{1.1}.(iii)}
	We now consider $0 \leq k<2$.
	Then the probability density of the random variable $|k + X_1 + X_1^{-1}|$ is given by
	\begin{equation*}
		p_1(k;x) = 
		\begin{cases}
			\frac{1}{2\pi} \frac{1}{\sqrt{1  - \frac{(x - k)^2}{4}}} &\textnormal{for } 2 - k \leq x < 2+k,\\
			\frac{1}{2\pi} \frac{1}{\sqrt{1  - \frac{(x - k)^2}{4}}} + \frac{1}{2\pi} \frac{1}{\sqrt{1  - \frac{(x + k)^2}{4}}}&\textnormal{for } 0 \leq x <2-k,\\
			0 &\textnormal{for }x<0 \textnormal{ or } x \geq 2 +k,
		\end{cases}
	\end{equation*}
	leading to
	\begin{align*}
		W_1(k;s) &= \int_{0}^{2+k} x^s p_1(k;x) \,\d x \\
		&= \frac{1}{2\pi}\int_{0}^{2+k} \frac{x^s}{\sqrt{1 - (\frac{x-k}{2})^2}} \,\d x + \frac{1}{2\pi}\int_{0}^{2-k} \frac{x^s}{\sqrt{1 - (\frac{x+k}{2})^2}} \,\d x.\\
	\end{align*}
	Substituting $y = \frac{x}{k+2}$, we obtain
	\begin{align*}
		\frac{1}{2 \pi} \int_{0}^{2+k} \frac{x^s}{\sqrt{1 - (\frac{x-k}{2})^2}} \,\d x &= \frac{1}{\pi}\int_{0}^{1} \frac{(2+k)^{s+1} y^s}{\sqrt{(1-y)(2+k)(y(2+k) + 2 - k)}} \, \d y\\
		&= \frac{(2+k)^{s + \frac{1}{2}}}{\pi(2 - k)^{\frac{1}{2}}} \int_{0}^{1} y^s(1 - y)^{-\frac{1}{2}}\left(1 - \frac{k+2}{k - 2}y \right)^{-\frac{1}{2}} \, \d y\\
		&= \frac{(2+k)^{s + \frac{1}{2}}}{\pi(2 - k)^{\frac{1}{2}}} \frac{\Gamma(s+1)\Gamma(\frac{1}{2})}{\Gamma(s + \frac{3}{2})} {}_2 F_{1} \left( \frac{1}{2}, s + 1; s + \frac{3}{2}; \frac{k+2}{k - 2}\right).
	\end{align*}
	
	Using this and the symmetric representation for $\int_{0}^{2-k}$, we deduce that
	\begin{align*}
		W_1(k;s) = \frac{\Gamma(s+1)}{\sqrt{\pi}\Gamma(s + \frac{3}{2})}  \biggl(\frac{(2+k)^{s + \frac{1}{2}}}{(2 - k)^{\frac{1}{2}}}&{}_2 F_{1} \left( \frac{1}{2}, s + 1; s + \frac{3}{2}; \frac{2+k}{-2 + k}\right) \\ & + \frac{(2-k)^{s + \frac{1}{2}}}{(2 + k)^{\frac{1}{2}}}{}_2 F_{1} \left( \frac{1}{2}, s + 1; s + \frac{3}{2}; \frac{-2 + k}{2 + k}\right) \biggr)
	\end{align*}
	for $\operatorname{Re}(s) > - 1$. 
	Applying the transformation 
	\begin{equation}
		\label{PFaff?}
		{}_{2}F_{1}(a,b;c;z)=(1-z)^{-a}\,{}_{2}F_{1} \left(a,c - b;c;{\frac {z}{z-1}} \right)
	\end{equation}
	valid for $|z| \leq \frac{1}{2}$,
	we find
	\begin{align*}
		W_1(k;s) = \frac{\Gamma(s+1)}{2 \sqrt{\pi}\Gamma(s + \frac{3}{2})} \biggl( (2 + k)^{s + \tfrac{1}{2}} & {}_2 F_{1} \left( \frac{1}{2},\frac{1}{2}; s + \frac{3}{2}; \frac{2 + k}{4}\right) \\& + (2 - k)^{s + \tfrac{1}{2}} {}_2 F_{1} \left( \frac{1}{2},\frac{1}{2}; s + \frac{3}{2}; \frac{2-k}{4}\right) \biggr).
	\end{align*}
	The latter expression can be simplified to the form \eqref{1.1.iii},
	using the quadratic transformation
	\begin{align*}
		\frac{2\Gamma(\frac{1}{2})\Gamma(a+b+\frac{1}{2})}{%
			\Gamma (a+\frac{1}{2})\Gamma (b+\frac{1}{2})}  {}_{2} F_{1}\left(a,b;%
		\tfrac{1}{2};z^2\right)&={}_{2} F_{1}\left(2a,2b;a+b+\tfrac{1}{2};\tfrac{1}{2}-\tfrac{1}{2}%
		z\right) \\& \qquad \qquad+ \, {}_{2} F_{1}\left(2a,2b;a+b+\tfrac{1}{2};\tfrac{1}{2}+\tfrac{1}{2}z\right).
	\end{align*}
	for $|z| \leq 1$ (see \cite[Eqn. (15.8.27)]{article}).
	%
	
	\section{Proof of Theorems \ref{1.2} and \ref{1.3}}
	In this section we discuss special properties of $W_r(k;s)$ for $r=1$. We give a functional equation for $W_1(k;s)$ and we state and prove the ``Riemann hypothesis'' for $s \mapsto W_1(k;s)$.
	It is not clear how much these results are generalizable to $r > 1$.
	
	\subsection{Functional Equations for $W_1(k;s)$}
	\subsubsection{Proof of Theorem \ref{1.2}.(i)}
	We start with the proof of Theorem \ref{1.2}.(i). We give two proofs of this fact. One proof uses hypergeometric transformations, and the other uses the probability distribution $p_1$.
	
	
	For the first proof, we use Euler's transformation formula
	\begin{equation}
		\label{EulerTransform}
		{}_2 F_{1} \left( a,b; c; z \right) = (1-z)^{c-a-b} {}_2 F_{1} \left( c-a,c -b; c; z \right),
	\end{equation}
	which is just the double iteration of the transformation \eqref{PFaff?}.
	We find out that
	\begin{align*}
		W_1(k;-s-1) &= |k|^{-s-1} {}_2 F_{1} \left( \frac{1+s}{2},1 + \frac{s}{2}; 1; \frac{4}{k^2} \right)\\
		&= |k|^{-s-1}\left(1 - \frac{4}{k^2}\right)^{ - s-\frac{1}{2}} {}_2 F_{1} \left( \frac{1-s}{2},-\frac{s}{2}; 1; \frac{4}{k^2} \right)\\
		&= (k^2 - 4)^{-s-\frac{1}{2}} W_1(k;s).
	\end{align*}
	For the second proof, we use the symmetry of  formula \eqref{p1}:
	\begin{equation}
		p_1 \left(k; \frac{k^2 - 4}{x} \right) = \frac{x}{\sqrt{k^2 - 4}} p_1(k;x)
	\end{equation}
	for $x \in (k-2,k+2)$.
	Now applying it to
	\begin{equation*}
		W_1(k;s) = \int_{|k|-2}^{|k|+2} x^s p_1(k;x) \, \d x
	\end{equation*}
	gives the functional equation \eqref{1.2.i}.
	\subsubsection{Proof of Theorem \ref{1.2}.(ii)}
	We start with the proof of Theorem \ref{1.2}.(ii).
	Using Euler's transformation formula \eqref{EulerTransform} we obtain
	\begin{align*}
		W_1(k;-s-1) &= \frac{4^{-s-1} \Gamma(-\frac{s}{2})^2}{\pi \Gamma(-s)} \cdot  {}_2 F_1 \left( \frac{1+s}{2}, \frac{1+s}{2}; \frac{1}{2} ; \frac{k^2}{4}\right)\\
		&= \frac{4^{-s-1} \Gamma(-\frac{s}{2})^2}{\pi \Gamma(-s)} \left( 1 - \frac{k^2}{4} \right)^{-s-\frac{1}{2}} {}_2 F_1 \left( -\frac{s}{2}, -\frac{s}{2}; \frac{1}{2} ; \frac{k^2}{4}\right)\\
		&= \frac{4^{-s-\frac{1}{2}}\Gamma(-\frac{s}{2})^2 \Gamma(1+s)}{\Gamma(-s) \Gamma(\frac{1+s}{2})^2} (4 - k^2)^{-s-\frac{1}{2}} W_1(k;s)\\
		&= 4^{-s-\frac{1}{2}} 2^{2s + 1} \cot\left(-\frac{\pi s}{2} \right) \cdot (4-k^2)^{-s-\frac{1}{2}} W_1(k;s)\\
		&= \cot\left(-\frac{\pi s}{2} \right) \cdot (4-k^2)^{-s-\frac{1}{2}} W_1(k;s),
	\end{align*}
	which is precisely the functional equation \eqref{1.2.ii}.

	\subsection{Zeros of $W_1(k;s)$}
	For $\lambda \in \mathbb{C}$ and $(\alpha,\beta) \in \mathbb{C}^2$, consider the Jacobi function 
	\begin{equation}
		\label{HGD}
		\phi_{\lambda}^{(\alpha,\beta)}(t) := (\cosh{t})^{-\alpha - \beta -1 - i\lambda} {}_2 F_1 \left( \frac{\alpha + \beta + 1 + i\lambda}{2},
		\frac{\alpha - \beta + 1 + i\lambda}{2}; \alpha + 1; \tanh^2{t}  \right),
	\end{equation}
	see \cite{Koornwinder}.
	When $(\alpha,\beta) \in \mathbb{C}^2$ is fixed, the sequence $\{ \phi^{(\alpha,\beta)}_{\lambda} \}_{\lambda \geq 0}$ forms a continuous orthogonal system on $\mathbb{R}_{\geq 0}$ with respect to the weight function
	\begin{equation*}
		\Delta_{\alpha,\beta}(t) := (2 \sinh{t})^{2 \alpha + 1} (2 \cosh{t})^{2 \beta + 1}.
	\end{equation*}
	In this way $\phi_{\lambda}^{(\alpha, \beta)}(t)$ becomes the unique even $C^{\infty}$-function $f$ on $\mathbb{R}$ satisfying
	
	\begin{equation}
		\label{DezeDE}
		\frac{\d^2f}{\d t^2} + \frac{\Delta'(t)}{\Delta(t)}\frac{\d f}{\d t} + \left(\lambda^2 + (\alpha + \beta + 1)^2 \right)f =0,
	\end{equation}
	where the dash denotes the derivative with respect to $t$.
	For brevity we write $\phi_{\lambda}$ for $\phi^{(\alpha,\beta)}_{\lambda}$, similarly for $\Delta$.
	\begin{lem}
		\label{Beauty}
		For any $x > 0$, $\lambda, \mu \in \mathbb{C}$ with $\lambda \neq \pm \mu$ and $\alpha, \beta \in \mathbb{R}$,
		\begin{equation}
			\label{BeautyEqn}
			\int_0^x \phi_\lambda(t) \phi_\mu(t) \Delta(t) \, \d t = (\mu^2 - \lambda^2)^{-1} \Delta(x) \left( \phi_\lambda'(x)\phi_\mu(x) - \phi_\lambda(x)\phi_\mu'(x) \right).
		\end{equation}
	\end{lem}
	\begin{proof}
		Note that we can rewrite \eqref{BeautyEqn} as $$\left(\Delta(t)\phi_{\lambda}'(t) \right)' = - \left(\lambda^2 + (\alpha + \beta + 1)^2 \right)\Delta(t)\phi_{\lambda}(t).$$
		Now performing integration by parts twice and using $\phi_\lambda(0) = 0$ gives
		\begin{align*}
			-(\mu^2 &+ (\alpha + \beta + 1)^2)\int_0^x \phi_\lambda(t) \phi_\mu(t) \Delta(t) \, \d t = \int_0^x \phi_\lambda(t) \left(\Delta(t)\phi_{\mu}'(t) \right)' \, \d t \\
			&=\Delta(x)\phi_\lambda(x)\phi_\mu'(x) -  \int_0^x \phi_\lambda'(t)\Delta(t) \phi_{\mu}'(t) \, \d t \\
			&=\Delta(x) \left( \phi_\lambda(x)\phi_\mu'(x) - \phi_\lambda'(x)\phi_\mu(x) \right)  + \int_0^x \left( \phi_\lambda'(t) \Delta(t) \right)' \phi_{\mu}(t) \, \d t \\
			&= \Delta(x) \left( \phi_\lambda(x)\phi_\mu'(x) - \phi_\lambda'(x)\phi_\mu(x) \right)\\ & \qquad -(\lambda^2 + (\alpha + \beta + 1)^2) \int_0^x \phi_\lambda(t) \phi_{\mu}(t) \Delta(t) \, \d t, 
		\end{align*}
		which is equation \eqref{BeautyEqn}.
	\end{proof}
	\begin{lem}
		\label{StelSTel}
		For any $x > 0$ and $\alpha, \beta \in \mathbb{R}$, if $\phi_\lambda(x) = 0$ then $\lambda \in \mathbb{R} \cup i \mathbb{R}$. Moreover, if $\min(\alpha+\beta+1,\alpha-\beta+1) > 0$ then $\phi_{i \mu}(x) > 0$ for $\mu \leq 0$.
	\end{lem}
	The idea of the proof of Lemma \ref{StelSTel} is based on the proof of Lommel's theorem on the zeros of Bessel function, see \cite[p. 482]{Watson}. 
	
	\emph{Proof of Lemma \ref{StelSTel}.}
	Fix $x > 0$ and assume that there is a $\lambda \not \in \mathbb{R} \cup i \mathbb{R}$ such that $\phi_\lambda(x) = 0$. Choose $\mu = \overline{\lambda}$ and apply Lemma \ref{Beauty}. Clearly, $\overline{\phi_\lambda} = \phi_\mu$, hence
	\begin{align*}
		\int_0^x |\phi_\lambda(t)|^2 \Delta(t) \, \d t = (\overline{\lambda}^2 - \lambda^2)^{-1} \Delta(x) \left( \phi_\lambda'(x)\overline{\phi_\lambda(x)} - \phi_\lambda(x)\overline{\phi_\lambda'(x)} \right) = 0.
	\end{align*}
	But this gives a contradiction as $\int_0^x |\phi_\lambda(t)|^2 \Delta(t) \, \d t$ is transparantly positive. Thus $\lambda \in \mathbb{R} \cup i \mathbb{R}$. If additionally $\min(\alpha+\beta+1,\alpha-\beta+1) > 0$, the coefficients of the hypergeometric function in \eqref{HGD} are strictly positive provided that $\lambda = i \mu$ for $\mu \leq 0$, showing that $\phi_{i\mu} >0$.
	\qed
	
	If $(\alpha, \beta) = (-\frac{1}{2},0)$ or $(0, -\frac{1}{2})$, then a consequence of this result is Theorem \ref{1.3}.
	\\
	\emph{Proof of Theorem \ref{1.3}.}
	Using Lemma \ref{StelSTel} it suffices to show that $s \mapsto W_1(k;s)$ has no \emph{real} zeros. For $s \geq 0$ this is obvious, as $W_1(k;s)$ coincides with an integral of a \emph{positive} function.
	\qed 
	\section{The zeta Mahler function $W_r(k;s)$ for general $r$}
	
	We start this section with noticing that for $r = 2$, the value of $W_2(k;s)$ coincides with $Z(k + x + x^{-1} + y + y^{-1}; s)$ (see \eqref{EQNN}). 
	Indeed, the substitution  $x = x_1x_2$ and $y = x_1x_2^{-1}$ in the latter leads to
	
	\begin{equation*}
		k + x_1x_2 + (x_1x_2)^{-1} + x_1x_2^{-1} + (x_1x_2^{-1})^{-1} =  k + (x_1 + x_1^{-1})(x_2 + x_2^{-1}).
	\end{equation*}
	For $r \geq 3$, the value of \eqref{TheIntegral} is different from $Z(k + x_1 + x_1^{-1} + \dots + x_r + x_r^{-1};s)$.

	
	In this section we discuss the densities $p_r(k;-)$.
	\subsection{The probability densities $p_r(k;-)$}
	In this section we explicitly compute the probability distributions $p_r(k;-)$ for $r = 1,2,3$ and discuss how to obtain them for any $r$.

	Define $\hat{p}_r$ to be the density of the random variable $(X_1 + X_1^{-1}) \cdots (X_r + X_r^{-1})$ on $\mathbb{C}^r$; note that the quantity assumes real values only. We can relate the density of $\hat{p}_r$ to $p_r(k;-)$ in the following way (compare with Section 2.2).
	\begin{lem}
		\label{HalloLemma}
		For $|k| < 2^r$, we have
		\begin{equation*}
			p_r(k;x) = 
			\begin{cases}
				\hat{p}_r(x - |k|)&\textnormal{for } 2^r - |k| \leq x < 2^r+|k|,\\
				\hat{p}_r(x - |k|) + \hat{p}_r(x + |k|)&\textnormal{for } 0 \leq x <2^r-|k|,\\
				0 &\textnormal{for }x<0 \textnormal{ or } x \geq 2^r +|k|.
			\end{cases}
		\end{equation*}
		For $|k| \geq 2^r$, we have $p_r(k;x) = \hat{p}_r(x - |k|)$.
	\end{lem}
	\begin{proof}
		We have 
		\begin{align*}
			W_r(k;s) &= \int_{\mathbb{R}} |z|^s \hat{p}_r(z-|k|) \, \d z\\ 
			&= \int_{2^r - |k|}^{2^r + |k|} |z|^s \hat{p}_r(z-|k|) \, \d z.
		\end{align*}
		If $|k| \geq 2^r$, then
		\begin{equation*}
			W_r(k;s) = \int_{2^r - |k|}^{2^r + |k|} z^s \hat{p}_r(z-|k|) \, \d z,
		\end{equation*}
		so that $p_r(k;x) = \hat{p}_r(x - |k|)$.
		\\
		If $|k| < 2^r$, then
		\begin{align*}
			W_r(k;s) &= \int_{2^r - |k|}^{2^r + |k|} |z|^s\hat{p}_r(z - |k|) \, \d z\\
			&= \int_0^{2^r + |k|} z^s\hat{p}_r(z - |k|) \, \d z + \int_0^{2^r - |k|} z^s\hat{p}_r(- z - |k|) \, \d z \\
			&= \int_0^{2^r + |k|} z^s\hat{p}_r(z - |k|) \, \d z + \int_0^{2^r - |k|} z^s\hat{p}_r(z + |k|) \, \d z\\
			&= \int_0^{2^r - |k|} z^s \left(\hat{p}_r(z - |k|) \, \d z + \hat{p}_r(z + |k|) \right) \, \d z + \int_{2^r - |k|}^{2^r + |k|} z^s\hat{p}_r(z - |k|) \, \d z,
		\end{align*}
		where the symmetry $\hat{p}_r(-z) = \hat{p}_r(z)$ was employed.
	\end{proof}
	Using Lemma \ref{HalloLemma}, it is clear that it suffices to compute the distributions $\hat{p}_r$.
	\subsubsection{Computation of $\hat{p}_r$}
	We will now compute $\hat{p}_r$ explicitly for $r = 1,2$ and $3$.
	
	It is easy to show (and well known) that  
	\begin{equation}
		\hat{p}_1(x) = \frac{1}{2\pi} \frac{1}{\sqrt{1 - \frac{x^2}{4}}},
	\end{equation}
	for $|x| < 2$ and $\hat{p}_1(x) = 0$ otherwise.
	
	For $r \geq 1$, we define $G_r$ via $\hat{p}_r(x) = G_r\left(1 - \frac{x^2}{4^r}\right).$
	For example,
	\begin{equation*}
		G_1(y) = \frac{1}{2 \pi} \frac{1}{\sqrt{y}} 
	\end{equation*}
	for $0 < y \leq 1$ and $G_1(y) = 0$ otherwise.

	Further, using basic properties of the Mellin transform, it follows that the $G_r$ satisfy a recurrence for $r \geq 2$:
	\begin{align}
		G_r \left(1 - \frac{x^2}{4^r} \right) &= \int_{\mathbb{R}} G_{r-1} \left(1-\frac{t^2}{4^{r-1}} \right)G_{1}\left(1 - \frac{x^2}{4t^2} \right) \frac{\d t}{|t|} \nonumber \\
		&= 2 \int_{x/2}^{2^{r-1}} G_{r-1} \left(1-\frac{t^2}{4^{r-1}} \right)G_{1}\left(1 - \frac{x^2}{4t^2} \right) \frac{\d t}{t} \nonumber \\
		&= \frac{1}{\pi} \int_{x/2}^{2^{r-1}} G_{r-1} \left(1 - \frac{t^2}{4^{r-1}} \right)\frac{\d t}{\sqrt{t^2 - \frac{x^2}{4}}} \label{DEZEDAN?}.
	\end{align}
	Then using the substitution $u = t^2/4^{r-1}$ in \eqref{DEZEDAN?}, we find
	
	\begin{align*}
		G_r \left(1 - \frac{x^2}{4^r} \right) &= \frac{1}{2 \pi} \int_{x^2/4^r}^{1} G_{r-1}(1 - u) \frac{\d u}{\sqrt{u(u - \frac{x^2}{4^r})}},\\
		&= \frac{1}{2 \pi} \int_{0}^{1 - x^2/4^r} G_{r-1}(u) \frac{\d u}{\sqrt{(1-u)(1 - \frac{x^2}{4^r} - u)}}.
	\end{align*}
	Finally, let $y = 1 - x^2/4^r$ and $v = u/y$, to arrive at the following result.
	\begin{thm}[Recursive formula for the $G_r$]
		For $G_r$ with $r \geq 2$ we have the following recursion:
		\begin{align}
			G_r(y) &= \frac{\sqrt{y}}{2\pi} \int_{0}^1 G_{r-1}(yv) \frac{\d v}{\sqrt{(1-v)(1-yv)}}.
		\end{align}
	\end{thm}
	
	Applying this recursion with $r = 2$ we obtain for $0 < y < 1$,
	\begin{align*}
		G_2(y) &= \frac{\sqrt{y}}{2\pi} \int_{0}^1 G_1 (yv) \frac{\d v}{\sqrt{(1-v)(1-yv)}}\\
		&= \frac{1}{4 \pi^2} \int_{0}^1 \frac{\d v}{\sqrt{v(1-v)(1-yv)}} \\
		& = \frac{1}{4 \pi} \cdot {}_2 F_{1} \left( \frac{1}{2}, \frac{1}{2};1;y \right).
	\end{align*}
	In the last step we use the integral representation \eqref{EulerType}.
	This shows that 
	\begin{equation}
		\hat{p}_2(x) = \frac{1}{4 \pi} \cdot {}_2 F_{1} \left( \frac{1}{2}, \frac{1}{2};1;1 - \frac{x^2}{16} \right),
	\end{equation}
	for $0 < |x| \leq 4$ and $\hat{p}_2(x) = 0$ otherwise.

	
	In the case $r = 3$ we proceed similarly. For $0 < y < 1$,
	
	\begin{align*}
		G_3(y) &= \frac{\sqrt{y}}{2 \pi} \int_0^1 G_2(yv) \frac{\d v}{\sqrt{(1-v)(1-yv)}} \\
		&=     \frac{\sqrt{y}}{8 \pi^2} \int_{0}^1 {}_2 F_{1} \left( \frac{1}{2}, \frac{1}{2};1; yv \right) \frac{\d v}{\sqrt{(1 - v)(1 - yv)}}.
	\end{align*}
	The expression 
	\begin{align*}
		Y(a) &:= \int_{0}^1 {}_2 F_{1} \left( \frac{1}{2}, \frac{1}{2};1; av \right) \frac{\d v}{\sqrt{(1 - v)(1 - av)}}\\
		&= \sum_{r \geq 0} \frac{r! \Gamma(1/2)}{\Gamma(r+3/2)} \left( \sum_{n + m = r} \frac{(\frac{1}{2})_n^2 (\frac{1}{2})_m}{(n!)^2 m!}\right) a^r
	\end{align*}
	satisfies a third-order linear differential equation
	\begin{equation}
		8a^2(a-1)^2\frac{d^3Y}{da^3} + 24a(a-1)(2a-1)\frac{d^2Y}{da^2} + (56a^2 - 54a + 6) \frac{dY}{da} + (8a - 3)Y = 0.
	\end{equation}
	After solving this differential equation (for example, with \emph{Mathematica} \cite{Mathema}) and checking the initial conditions we find out that
	\begin{equation}
		G_3(y) = \frac{\sqrt{y}}{4 \pi^2} {}_2 F_{1} \left(\frac{1}{4}, \frac{1}{4} ;\frac{1}{2};y\right) {}_2 F_{1} \left(\frac{3}{4}, \frac{3}{4} ;\frac{3}{2};y\right),
	\end{equation}
	so that
	\begin{equation}
		\hat{p}_3(x) = \frac{\sqrt{1 - \frac{x^2}{64}}}{4 \pi^2} {}_2 F_{1} \left(\frac{1}{4}, \frac{1}{4} ;\frac{1}{2};1 - \frac{x^2}{64}\right) {}_2 F_{1} \left(\frac{3}{4}, \frac{3}{4} ;\frac{3}{2};1 - \frac{x^2}{64}\right)
	\end{equation}
	for $0< |x| \leq 8$ and $\hat{p}_3(x) = 0$ for $|x| > 8$.
	
	Already at this stage it is pretty suggestive that $G_r(1-y)$ satisfies the same differential equation as 
	$${}_r F_{r-1} \left( \frac{1}{2}, \dots ,  \frac{1}{2};1, \dots , 1; y \right).$$
	We discuss this in the next subsection.
	

	\subsubsection{Mellin Transform of $\hat{p}_r$}
	In this section we find an expression for 
	\begin{equation*}
		\int_{-2^r}^{2^r} x^{v} \hat{p}_r(x) \, \d x
	\end{equation*}
	when $v \in \mathbb{C}$.
	This is the $v$-th moment of the random variable $(X_1 + X_1^{-1})  \dots  (X_r + X_r^{-1})$.
	\begin{lem}
		\label{Grappig}
		For $\textnormal{Re}(v) > -1$, we have
		\begin{equation}
			\int_0^{2^r} x^v \hat{p}_r(x) \, \d x =\frac{1}{2}\left(\frac{2^v}{\pi} \frac{\Gamma\left(\frac{1}{2} \right) \Gamma(\frac{v+1}{2})}{\Gamma\left(1 + \frac{v}{2} \right)}  \right)^r
		\end{equation}
		and 
		\begin{equation}
			\int_{-2^r}^{2^r} x^v \hat{p}_r(x) \, \d x = \frac{1+e^{\pi i v}}{2}\left(\frac{2^v}{\pi} \frac{\Gamma\left(\frac{1}{2} \right) \Gamma(\frac{v+1}{2})}{\Gamma\left(1 + \frac{v}{2} \right)}  \right)^r.
		\end{equation}
	\end{lem}
	\begin{proof}
		Since $$|(X_1 + X_1^{-1})  \cdots  (X_r + X_r^{-1})|^v =|X_1 + X_1^{-1}|^v \cdots |X_r + X_r^{-1}|^v $$
		and for the expected value
		\begin{equation*}
			\mathsf{E}[|(X_1 + X_1^{-1}) \cdots (X_r + X_r^{-1})|^v] = 2 \int_0^{2^r} x^v \hat{p}_r(x) \, \d x,
		\end{equation*}
		we have
		\begin{equation*}
			2 \int_{0}^{2^r} x^v \hat{p}_r(x) \, \d x  = \left(2 \int_{0}^{2} x^v \hat{p}_1(x) \, \d x \right)^r.
		\end{equation*}
		For $r = 1$, we obtain
		\begin{align*}
			\int_{0}^2 x^v \hat{p}_1(x) \, \d x &= \frac{1}{2 \pi} \int_{0}^2 \frac{x^v}{\sqrt{1 - \frac{x^2}{4}}} \, \d x \\
			&=\frac{2^{v}}{\pi} \int_{0}^1 \frac{x^v}{\sqrt{1 - x^2}} \, \d x\\
			&=\frac{2^{v}}{\pi}  \frac{\Gamma(v+1)\Gamma(\frac{1}{2})}{\Gamma(\frac{3}{2}+v)} \cdot {}_2 F_{1}  \left(\frac{1}{2} ,v+1;\frac{3}{2} + v; -1 \right) \\
			&= \frac{ 2^{v}}{\pi} \frac{\Gamma(\frac{1}{2}) \Gamma(1 + \frac{v+1}{2})}{(v+1) \Gamma(1 + \frac{v}{2})}.
		\end{align*}
		Therefore,
		\begin{equation*}
			\int_0^{2^r} x^v \hat{p}_r(x) \, \d x = \frac{1}{2} \left(\frac{2^{v+1}}{\pi}   \frac{\Gamma(\frac{1}{2}) \Gamma(1 + \frac{v+1}{2})}{(v+1) \Gamma(1 + \frac{v}{2})}  \right)^r
		\end{equation*}
		and, clearly,
		\begin{equation*}
			\int_{-2^r}^{2^r} x^v \hat{p}_r(x) \, \d x = \left( 1 + e^{\pi i v} \right) \int_{0}^{2^r} x^v \hat{p}_r(x) \, \d x.  \qedhere
		\end{equation*}
	\end{proof}

	Let $f \colon (0, \infty) \to \mathbb{R}$ be continuously differentiable and $s \in \mathbb{C}$. Denote by $M(f;s)$ its Mellin transform 
	\begin{equation*}
		M(f;s) := \int_0^\infty x^s f(x) \, \d x.
	\end{equation*}
	Note that our definition of the Mellin transform differs slightly from the standard one, this does not affect the properties discussed below.
	Further, let $\theta$ be the differential operator $x \frac{\d}{\d x}$.
	\begin{prop}[General properties of Mellin transforms; see {\cite[Section 3.1.2]{Paris}}]
		\label{PROPROP}
		\
		\begin{enumerate}
			\item[(i)] $M(\theta f(x);s) = -(s+1) M(f(x);s)$;
			\item[(ii)] $M(xf(x);s) = M(f(x);s+1) $.
		\end{enumerate} 
	\end{prop}
	
	Let $H_r(y) := G_r(1-y)$ and consider the Mellin transform
	\begin{align*}
		M(H_r;s) = \int_0^\infty y^s H_r(y) \, \d y.    
	\end{align*}
	Substituting $y = x^2/4^r$ and using Lemma \ref{Grappig}, we obtain
	\begin{align*}
		M(H_r;s) =\int_0^1 y^s H_r(y) \, \d y &= \frac{2}{4^{r(s+1)}} \int_0^{2^r} x^{2s+1} \hat{p}_r(x) \, \d x \\   &=\frac{1}{4^{r(s+1)}} \left(\frac{1}{\pi}  2^{2s+1} \frac{\Gamma \left(\frac{1}{2} \right) \Gamma \left(\frac{(2s+1)+1}{2}\right)}{\Gamma \left(1 + \frac{2s+1}{2} \right)}  \right)^r\\
		&= \left(\frac{\Gamma \left(\frac{1}{2} \right)\Gamma(s+1)}{2\pi \Gamma \left(s+\frac{3}{2}\right)} \right)^r
		.
	\end{align*}
	It is clear that $M(H_r;s)$ satisfies a recursion:
	\begin{equation*}
		\left(s+\frac{1}{2}\right)^r M(H_r;s) = s^r M(H_r;s-1).
	\end{equation*}
	If $H_r$ were continuously differentiable, it would follow that 
	\begin{align*}
		s^r M(H_r;s-1) = (-1)^r M(\theta^r H_r;s-1)
	\end{align*}
	and
	\begin{align*}
		\left(s+\frac{1}{2} \right)^r M(H_r;s) = (-1)^r M \left(\left(\theta + \frac{1}{2} \right)^r H_r;s \right).
	\end{align*}
	In other words, $$M(\theta^r H_r;s-1) = M \left(\left(\theta + \frac{1}{2} \right)^r H_r;s \right),$$
	so that $H_r(y)$ is annihilated by the differential operator $\theta^{r} - y(\theta + 1/2)^r$. This means that $H_r$ satisfies the same differential equation as ${}_r F_{r-1} \left( \frac{1}{2}, \dots ,  \frac{1}{2};1, \dots , 1; y \right)$. Though $H_r$ is not continuously differentiable, giving a similar argument as in the proof of \cite[Theorem 2.4]{borwein_straub_wan_zudilin_zagier_2012} it follows rigorously that $H_r$ satisfies the same differential equation in a distributional sense. 
	
	
	
	
	\subsection{$W_r$ for $|k| \geq 2^r$}
	In this part, we compute the value of $W_r(k;s)$ for $|k| \geq 2^r$.
	
	Define for $s \in \mathbb{C}$ the following function of $z \in \mathbb{C}$:
	\begin{equation}
		F_{r,s}(z) := \int_{z - 2^r}^{z + 2^r} x^s \hat{p}_r(x- z) \, \d x.
	\end{equation}
	
	\begin{prop}[Properties of $F_{r,s}$]
		\label{AnalProp}
		\
		\begin{enumerate}
			\item[(i)]  $F_{r,s}$ is analytic on $\mathbb{C} \setminus (-\infty, 2^r]$. 
			\item[(ii)] $F_{r,s}$ is continuous on $\{ z \in \mathbb{C} \ \colon   \operatorname{Im}(z) \geq 0\}$.
		\end{enumerate}
	\end{prop}
	\begin{proof}
		This follows immediately as we can write \begin{equation*}
			F_{r,s}(z) = \int_{- 2^r}^{2^r} (x + z)^s \hat{p}_r(x) \, \d x
		\end{equation*}
		and notice that $z \mapsto (x + z)^s$ is holomorphic for $z \not \in (-\infty, 2^r]$.
	\end{proof}
	
	%
	%
	First of all, we can find an explicit expression for $F_{r,s}(z)$ if $|z|>2^r$. The value of $F_{r,s}(z)$ when $|z| \leq 2^r$ will follow from the analytic continuation of $F_{r,s}$, with the help of Proposition \ref{AnalProp}.
	\begin{lem}
		\label{Bel_angrijk}
		For $|z| > 2^r$, we have $$F_{r,s}(z) = z^s \cdot {}_{r+1} F_{r} \left( \frac{-s}{2}, \frac{1-s}{2}, \frac{1}{2}, \dots , \frac{1}{2}; 1, \dots, 1; 
		\frac{4^r}{z^2} \right). $$
	\end{lem}
	\begin{proof}
		Note that
		\begin{align*}
			F_{r,s}(z) &= \int_{-2^r}^{2^r} (x+z)^s \hat{p}_r(x) \, \d x, \\
			&= z^s\int_{-2^r}^{2^r} \left(1+ \frac{x}{z} \right)^s \hat{p}_r(x) \, \d x,
		\end{align*}
		as $\operatorname{Re}\left(1+ \frac{x}{z} \right) > 0$.
		Since $|x/z|<1$, we can write
		\begin{equation*}
			\left(1 + \frac{x}{z} \right)^s = \sum_{n \geq 0} \binom{s}{n} \frac{x^n}{z^n},
		\end{equation*}
		so that
		\begin{align*}
			F_{r,s}(z) = z^s \sum_{n \geq 0} \binom{s}{n} \frac{1}{z^n} \int_{-2^r}^{2^r} x^n  \hat{p}_r(x) \, \d x.
		\end{align*}
		Now using Lemma \ref{Grappig}, we find
		\begin{align*}
			\int_{-2^r}^{2^r} x^n \hat{p}_r(x) \, \d x &= \frac{1+(-1)^n}{2}\left(\frac{2^n}{\pi}  \frac{\Gamma(\frac{1}{2}) \Gamma(\frac{n+1}{2})}{\Gamma(1 + \frac{n}{2})}  \right)^r \\
			&= \begin{cases} \binom{n}{n/2}^r \qquad &\textnormal{if $n$ is even,} \\ 0 \qquad &\textnormal{otherwise.} \end{cases}
		\end{align*}
		Thus,
		\begin{align*}
			F_{r,s}(z) &= z^s \sum_{n \geq 0} \binom{s}{2n} \binom{2n}{n}^r \frac{1}{z^{2n}}\\
			&= z^s \sum_{n \geq 0} \binom{s}{2n} \left(\frac{(\frac{1}{2})_n \cdot 4^n}{n!} \right)^r \frac{1}{z^{2n}}\\
			&= z^s \sum_{n \geq 0} \frac{(\frac{1-s}{2})_n (\frac{-s}{2})_n}{n!^2} \left(\frac{(\frac{1}{2})_n}{n!} \right)^{r-1} \left(\frac{4^r}{z^{2}}\right)^n\\
			&=  z^s \cdot {}_{r+1} F_{r} \left( \frac{-s}{2}, \frac{1-s}{2}, \frac{1}{2}, \dots , \frac{1}{2}; 1, \dots, 1; 
			\frac{4^r}{z^2} \right). \qedhere
		\end{align*}
	\end{proof}
	Now for $|k| > 2^r,$ $k \in \mathbb{R}$ we can deduce an explicit formula for $W_r$.
	
	\emph{Proof of Theorem \ref{1.4}.(i)}.
	Clearly,
	\begin{equation*}
		W_r(k;s) = F_{r,s}(|k|),
	\end{equation*} so that
	\begin{equation*} 
		W_r(k;s) =  |k|^s \cdot {}_{r+1} F_{r} \left( \frac{-s}{2}, \frac{1-s}{2}, \frac{1}{2}, \dots, \frac{1}{2}; 1, \dots, 1; 
		\frac{4^r}{k^2} \right),
	\end{equation*}
	where Lemma \ref{Bel_angrijk} was used. \qed

	We see that, for a fixed $|k| > 2^r$, the mapping $s \mapsto W_r(k;s)$ defines an entire function.
	\begin{remark}
		If $s$ is a non-negative integer, the hypergeometric sum is terminating. Hence, $W_r(k;s)$ becomes a polynomial in $|k|$.
	\end{remark}
	The value of $W_r(k;s)$ for $k = \pm 2^r$ can be found by taking the limit.

	\begin{proof}[Proof of Theorem \ref{1.4}.(ii)]
		This follows from the absolute convergence of 
		\begin{equation*}
			{}_{r+1} F_r \left( a_1, \dots, a_{r+1};b_1, \dots, b_r;z \right)
		\end{equation*}
		on the domain $|z| \leq 1$ if
		\begin{equation*}
			\operatorname{Re}\left(\sum_{i =1}^{r} b_i -  \sum_{j=1}^{r+1} a_j \right) >0
		\end{equation*}
		(see \cite[p. 156]{MR0501762}). \end{proof}
\begin{exmp}
	For $r = 1$ and $|k|>2$, equation \eqref{1.4.i} gives
	\begin{equation}
		W_1(k;s) = |k|^s \cdot {}_{2} F_{1} \left( \frac{-s}{2}, \frac{1-s}{2} ; 1;     \frac{4}{k^2} \right)
	\end{equation}
	and for $|k| = 2$, equation \eqref{1.4.ii} gives
	\begin{equation}
		W_1(k;s) = \frac{2^s \Gamma(\frac{1}{2} + s)}{\Gamma(1+\frac{s}{2}) \Gamma(\frac{1+s}{2})},
	\end{equation}
	for $\operatorname{Re}(s) > -\frac{1}{2}$.
	This leads to a second proof of Theorem \ref{1.1}.
\end{exmp}

\subsection{$W_r$ for $|k| < 2^r$}



For $|k|<2^r$, we write
\begin{equation}
	\label{DezeVGL}
	W_r(k;s) = \int_{0}^{|k|+2^r} x^s \hat{p}_r(x-|k|) \, \d x + \int_{0}^{-|k| + 2^r} x^s \hat{p}_r(x+|k|) \, \d x.
\end{equation}
Note that 
\begin{equation*}
	\int_{|k|- 2^r}^{|k|+2^r} x^s \hat{p}_r(x-|k|) \, \d x = \int_{0}^{|k|+2^r} x^s \hat{p}_r(x-|k|) \, \d x + e^{\pi i s} \int_{0}^{-|k|+2^r} x^s \hat{p}_r(x+|k|) \, \d x,
\end{equation*}
hence
\begin{align}
	W_r(k;s) &= \frac{1}{1+e^{\pi i s}} \left(\int_{|k| - 2^r}^{|k|+2^r} x^s \hat{p}_r(x-|k|) \, \d x  + \int_{-|k| - 2^r}^{-|k|+2^r} x^s \hat{p}_r(x+|k|) \, \d x   \right) \nonumber \\
	&= \frac{1}{1+e^{\pi i s}} \left( F_{r,s}(|k|) + F_{r,s}(-|k|) \right). \label{DDD}
\end{align}

We now define the following analytic continuations $F_{r,s}^{+}$ and $F_{r,s}^{-}$.
Let $F_{r,s}^{+}(z)$ be the analytic continuation of the function 
$$z^s \cdot {}_{r+1} F_{r} \left( \frac{-s}{2}, \frac{1-s}{2}, \frac{1}{2}, \dots , \frac{1}{2}; 1, \dots , 1; 
\frac{4^r}{z^2} \right)$$ on the complement of the upper half-disk $D(0,2^r)^\mathsf{c} \cap \mathbb{H}^{+}$ to $\mathbb{C} \setminus ((-\infty,0] \cup [2^r, \infty))$, while
$F_{r,s}^{-}(z)$ is the analytic continuation of the function $$(-z)^s \cdot {}_{r+1} F_{r} \left( \frac{-s}{2}, \frac{1-s}{2}, \frac{1}{2}, \dots , \frac{1}{2}; 1, \dots , 1; 
\frac{4^r}{z^2} \right)$$ on the complement of the lower half-disk $D(0,2^r)^\mathsf{c} \cap \mathbb{H}^{-}$ to $\mathbb{C} \setminus ((-\infty,0] \cup [2^r, \infty))$.
Here $\mathbb{H}^+$ and $\mathbb{H}^{-}$ denote the (strict) upper- and lower-half planes, respectively.

Using Proposition \ref{AnalProp}, it is clear that for $|k| < 2^r$ we have
\begin{align*}
	F_{r,s}(|k|) = F_{r,s}^{+}(|k|) \quad \text{ and } \quad F_{r,s}(-|k|) = F_{r,s}^{-}(|k|).
\end{align*}
Define
\begin{equation*}
	H_{r,s}(z) := \frac{1}{1+e^{\pi i s}} \left(F_{r,s}^{+}(z) + F_{r,s}^{-}(z) \right),
\end{equation*}
which is analytic on $\mathbb{C} \setminus ((-\infty,0] \cup [2^r, \infty))$, with the following motivation.
\begin{thm}
	\label{StellingHW}
	For $|k| < 2^r$,
	\begin{equation}
		W_r(k;s) = H_{r,s}(|k|).
	\end{equation}
\end{thm}
In fact, in the case of real $s$, we have an additional structure.
\begin{lem}
	\label{Thereallemma}
	For real positive $s$, we have
	\begin{equation*}
		\overline{F_{r,s}(-|k|)} = e^{-\pi i s}F_{r,s}(|k|).
	\end{equation*}
\end{lem}
\begin{proof}
	For notational convenience assume $0 \leq k < 2^r$. Then
	\begin{align*}
		F_{r,s}(-k) = \int_{k}^{2^r} (x-k)^s \hat{p}_r(x)  \,\d x + e^{\pi i s} \int_{-k}^{2^r} (x+k)^s \,\d x
	\end{align*}
	and taking the complex conjugate, we deduce that
	\begin{align*}
		\overline{F_{r,s}(-k)} &= e^{-\pi i s}\int_{-2^r}^{-k} (x+k)^s \hat{p}_r(x)  \,\d x + e^{-\pi i s} \int_{-k}^{2^r} (x+k)^s \,\d x\\
		&= e^{-\pi i s}F_{r,s}(k),
	\end{align*}
	which is the desired claim.
\end{proof}
Now we establish the proof of Theorem \ref{1.5}.
\begin{proof}[Proof of Theorem \ref{1.5}.] Observe that for real $s$ Theorem \ref{1.4} coincides with the statement of Theorem \ref{1.5}, as $${}_{r+1} F_{r} \left( \frac{-s}{2}, \frac{1-s}{2}, \frac{1}{2}, \dots , \frac{1}{2}; 1, \dots , 1; 
	\frac{4^r}{k^2} \right)$$ is real-valued for $|k| \geq 2^r$. Therefore, we can assume $|k| < 2^r$. Then by Lemma \ref{Thereallemma} and Theorem \ref{StellingHW} we obtain
	\begin{align*}
		W_r(k;s) &= \frac{1}{1+e^{\pi i s}} (F_{r,s}(|k|) + e^{\pi i s}\overline{F_{r,s}(|k|)}) \\
		&= \operatorname{Re} F_{r,s}(|k|) + \tan\left(\frac{\pi s}{2} \right)  \operatorname{Im} F_{r,s}(|k|). \qedhere
	\end{align*}
\end{proof}
\begin{remark}
	Note that $F_{r,s}^{+}(z), F_{r,s}^{-}(z)$ and hence also $H_{r,s}(z)$ satisfy the same differential equation as $${}_{r+1} F_{r} \left( \frac{-s}{2}, \frac{1-s}{2}, \frac{1}{2}, \dots , \frac{1}{2}; 1, \dots , 1;
	\frac{4^r}{z^2} \right).$$\end{remark}
It is clear from the definition that $F_{r,s}^{+}(z) = F_{r,s}^{-}(-z)$ for $z \in D(0,2^r)^\mathsf{c} \cap \mathbb{H}^{+}$; therefore, from the theory of analytic continuation we conclude with the following.
\begin{prop}
	\label{Propperdeprop}
	For all $z \in \mathbb{H}^{+},$ 
	$$F_{r,s}^{+}(z) = F_{r,s}^{-}(-z).$$
\end{prop}


\begin{prop}[Real differentiability at $k = 2^r$]
	\label{VeryUsefull}
	Suppose $\operatorname{Re}(s) > n - \frac{r}{2}$. Then $k \mapsto W_r(k;s)$ is $n$ times (real) differentiable at $k = 2^r$.
\end{prop}
\begin{proof}
	We will prove that $k \mapsto W_r(k;s)$ is differentiable at $k = 2^r $ if $\operatorname{Re}(s) > \frac{1}{2}$. The general case follows by induction.
	For $k > 2^r$, we have $W_r(k;s) = F_{r,s}(k)$, so that the right derivative at $k = 2^r$ is given by
	\begin{equation*}
		\int_{-2^r}^{2^r} s(x + 2^r)^{s-1} \hat{p}_r(x) \, \d x = s  F_{r,s-1}(2^r)
	\end{equation*}
	when $\operatorname{Re}(s) > 1 - \frac{r}{2}$. 
	For $|k|< 2^r$, we have $$W_r(k;s) = \frac{1}{1 + e^{\pi i s}} \left(F_{r,s}(k) + F_{r,s}(-k) \right).$$ Therefore, the left derivative at $k = 2^r$ is given by
	\begin{equation*}
		\frac{1}{1+e^{\pi i s}} \left(s F_{r,s-1}(2^r) + e^{\pi i s} sF_{r,s-1}(2^r) \right) = s  F_{r,s-1}(2^r).
	\end{equation*}
	This means that at $k = 2^r$ the left and right derivatives coincide.
\end{proof}
\begin{prop}[Real derivatives at $k = 0$]
	\label{VeryUsefull2}
	The right real derivatives $(\d ^j/\d k^j)^{+}$ for $0 \leq j \leq \floor*{\operatorname{Re}(s)}$ of $k \mapsto W_r(k;s)$ at $k=0$ are given by
	\begin{equation}
		\frac{\d ^jW_r(k;s)}{\d k^j}^{+}\Biggr|_{k = 0} = 
		\begin{cases}
			0 \qquad &\textnormal{if $j$ is odd,}\\
			\frac{\Gamma(s+1)^{r+1} }{\Gamma(s-j+1) \Gamma(1+\frac{s}{2})^{2r}} \qquad &\textnormal{if $j$ is even.}
		\end{cases}    
	\end{equation}
\end{prop}
\begin{proof}
	For $k \geq 0$, we have
	\begin{equation*}
		\frac{\d ^jW_r(k;s)}{\d k^j}^{+} = s(s-1)\frac{\d^{j-2} W_r(k;s)}{\d k^{j-2}}^{+}.
	\end{equation*}
	By induction,
	\begin{align*}
		\frac{\d^jW_r(k;s)}{\d k^j}^{+} &= s(s-1) \cdots (s-j+1) W_r(k;s-j)\\
		&= \frac{\Gamma(s+1)}{\Gamma(s-j+1)}W_r(k;s-j)
	\end{align*}
	for even $j$.
	Finally, notice that
	\begin{equation*}
		\frac{\d W_r(k;s)}{\d k}^{+}\Bigr|_{k = 0} = 0. \qquad \qedhere
	\end{equation*}
\end{proof}
We now return to Theorem \ref{1.1}.(iii) and give another proof of it.
\begin{proof}[Proof of Theorem \ref{1.1}.(iii)]
	For $|z|>2$, we have that 
	\begin{equation*}
		F_{1,s}(z) = z^s \cdot {}_{2} F_{1} \left( \frac{-s}{2}, \frac{1-s}{2} ; 1; \frac{4}{z^2} \right)
	\end{equation*}
	implying that
	$F_{1,s}$ is a solution to the differential equation
	\begin{equation}
		\label{DezeEQ}
		\left(-\frac{1}{4}z^2+1 \right)\frac{dY^2}{dz^2}+\left(\frac{1}{2}sz - \frac{1}{4}z\right)\frac{dY}{dz} -\frac{1}{4}s^2Y = 0.
	\end{equation}
	The equation
	\eqref{DezeEQ} has a basis of solutions
	\begin{align*}
		Y_0(z) &= {}_{2} F_{1} \left( \frac{-s}{2}, \frac{-s}{2} ; \frac{1}{2}; \frac{z^2}{4} \right),\\
		Y_1(z) &= z \cdot {}_{2} F_{1} \left( \frac{1-s}{2}, \frac{1-s}{2} ; \frac{3}{2}; \frac{z^2}{4} \right),
	\end{align*}
	so that $F_{1,s}^{+}(z) = C_0 Y_0(z) + C_1 Y_1(z)$, $F_{1,s}^{-}(z) = \tilde{C_0} Y_0(z) +  \tilde{C_1} Y_1(z)$ and $H_{1,s}(z) = D_0 Y_0(z) + D_1 Y_1(z)$ for constants $C_0, C_1, \tilde{C_0}, \tilde{C_1}, D_0$ and $D_1$ depending only on $s$.
	Using Proposition \ref{Propperdeprop}, it follows that $C_1 = -\tilde{C_1}$,
	hence $D_1 = 0$.
	
	As $\lim_{z \to 0} H_{1,s}(z) = W_1(0;s)$ and
	we have $$W_1(0;s) = \frac{2^{s} \Gamma(\frac{1}{2}) \Gamma(\frac{s+1}{2})}{\pi \Gamma(1+\frac{s}{2})},$$
	it follows that \begin{equation*}
		H_{1,s}(z) =  \frac{2^{s} \Gamma \left(\frac{1}{2}\right) \Gamma \left(\frac{s+1}{2}\right)}{\pi \Gamma\left(1+\frac{s}{2}\right)} \cdot {}_{2} F_{1} \left( \frac{-s}{2}, \frac{-s}{2} ; \frac{1}{2}; \frac{z^2}{4} \right).
	\end{equation*}
	Thus,
	$$W_{1}(k;s) = \frac{4^s \Gamma(\frac{1 + s}{2})^2}{\pi \Gamma(1 + s)} \cdot {}_{2} F_{1} \left( \frac{-s}{2}, \frac{-s}{2} ; \frac{1}{2}; \frac{k^2}{4} \right)$$
	for $|k|<2^r$.
\end{proof}

\begin{remark}
	In the proof above we see that $H_{1,s}(z)$ extends to an analytic function in a neighborhood of $z = 0$. We expect that this will only happen in the case $r = 1$.
\end{remark}
\begin{remark}
	We use the symmetry of $H_{1,s}$ to show that $D_1 = 0$. Clearly, using $$\lim_{z \to 2} H_{1,s}(z) = W_1(2;s) = 2^s \cdot {}_{2} F_{1} \left( \frac{-s}{2}, \frac{1-s}{2}; 1 ;1 \right)$$
	would lead to the same conclusion.
\end{remark}

For $r=2$ we follow the above strategy of the case $r = 1$.
\begin{proof}[Proof of Theorem \ref{1.6}]
	We have, for $|z|>4$,
	\begin{equation*}
		F_{2,s}(z) = z^s \cdot {}_{3} F_{2} \left( \frac{-s}{2}, \frac{1-s}{2}, \frac{1}{2} ; 1, 1; \frac{16}{z^2} \right),
	\end{equation*}
	hence $F_{2,s}$ is a solution to the differential equation 
	\begin{equation}
		\label{DiffEqr2}
		\left(-\frac{1}{8}z^3+2z\right)\frac{\d ^3Y}{\d z^3} + \left(\frac{3}{8}sz^2 - \frac{3}{8}z^2 - 2s + 2 \right)\frac{\d^2Y}{\d z^2}+\left(-\frac{3}{8}s^2z + \frac{3}{8}sz - \frac{1}{8}z \right)\frac{\d Y}{\d z} + \frac{1}{8}s^3Y =0.
	\end{equation}
	A basis of solutions for \eqref{DiffEqr2} is given by
	\begin{align*}
		Y_0(z) &=  {}_{3} F_{2} \left( \frac{-s}{2}, \frac{-s}{2}, \frac{-s}{2} ; \frac{1-s}{2}, \frac{1}{2}; \frac{z^2}{16} \right),\\
		Y_1(z) &= z \cdot {}_{3} F_{2} \left( \frac{1-s}{2}, \frac{1-s}{2}, \frac{1-s}{2} ; 1-\frac{s}{2}, \frac{3}{2}; \frac{z^2}{16} \right), \\
		Y_2(z) &= z^{1+s} \cdot {}_{3} F_{2} \left( \frac{1}{2}, \frac{1}{2}, \frac{1}{2} ; 1+\frac{s}{2}, \frac{3}{2} + \frac{s}{2}; \frac{z^2}{16} \right).
	\end{align*}
	It follows that $H_{2,s}(z) = D_0 Y_0(z) + D_1 Y_1(z) + D_2 Y_2(z)$ for some constants $D_0,D_1$ and $D_2$ depending only on $s$. Using the same argument as in the second proof of Theorem \ref{1.5}, it follows that $D_1 = 0$.
	Since 
	$$\lim_{z \to 0}H_{2,s}(z) = W_{2}(0;s) = \frac{\Gamma(s+1)^2}{\Gamma(\frac{s}{2}+1)^4}$$
	and
	$$\lim_{z \to 4}H_{2,s}(z) = W_{2}(4;s) = 4^s \cdot {}_{3} F_{2} \left( \frac{-s}{2}, \frac{1-s}{2}, \frac{1}{2} ; 1, 1; 1 \right),$$
	we find out that 
	\begin{equation*}
		\begin{split}
			H_{2,s}(z) =&\frac{\Gamma(s+1)^2}{\Gamma(\frac{s}{2}+1)^4}  \cdot Y_0(z) \\&-\frac{4^s {}_3 F_{2} \left(\frac{1}{2},\frac{-s}{2},\frac{1-s}{2};1,1;1 \right) -
				\frac{\Gamma(s+1)^2}{\Gamma(\frac{s}{2}+1)^4} {}_3 F_{2}  \left(\frac{-s}{2} ,\frac{-s}{2},\frac{-s}{2}; \frac{1-s}{2},\frac{1}{2}; 1 \right) }{4^{1+s} \cdot {}_3 F_{2} \left(\frac{1}{2},\frac{1}{2} ,\frac{1}{2};1 +\frac{s}{2}, \frac{3}{2} + \frac{s}{2} ; 1 \right)} \cdot Y_2(z).\\
		\end{split}
	\end{equation*}
	
	We can further simplify the coefficient in front of $Y_2$.
	We have the following hypergeometric identity for the special value at $z=1$ (see \cite{bailey1935generalized}):
	\begin{align*}
		{}_3 F_{2} \left(a_1,a_2,a_3;b_1,b_2;1 \right)          =&\frac{\Gamma(1-a_2)\Gamma(a_3-a_1)\Gamma(b_1)\Gamma(b_2)}{\Gamma(a_1-a_2+1)\Gamma(a_3)\Gamma(b_1-a_1)\Gamma(b_2-a_1)} \\& \quad \times {}_3 F_{2} \left(a_1,a_1-b_1+1,a_1-b_2+1;a_1-a_2+1,a_1-a_3+1;1 \right)\\
		&+\frac{\Gamma(1-a_2)\Gamma(a_1-a_3)\Gamma(b_1)\Gamma(b_2)}{\Gamma(a_1)\Gamma(a_3-a_2+1)\Gamma(b_1-a_3)\Gamma(b_2-a_3)} \\& \quad \times {}_3 F_{2} \left(a_3,a_3-b_1+1,a_3-b_2+1;a_3-a_1+1,a_3-a_2+1;1 \right).
	\end{align*}
	Applying this identity with $a_1 = 1/2, a_2 = 1/2-s/2,
	a_3 = -s/2,
	b_1 = 1$ and $
	b_2 = 1$ gives
	
	\begin{equation*}
		\begin{split}
			{}_3 F_{2} \left(\frac{1}{2},\frac{1-s}{2},\frac{-s}{2};1,1;1 \right) 
			=\frac{\Gamma(\frac{1}{2}+\frac{s}{2})\Gamma(\frac{-s}{2}-\frac{1}{2})}{\Gamma(\frac{s}{2}+1)\Gamma(\frac{-s}{2}) \pi} & {}_3 F_{2} \left(\frac{1}{2},\frac{1}{2},\frac{1}{2};1+\frac{s}{2},\frac{3}{2}+\frac{s}{2};1 \right)\\
			+\frac{\Gamma(\frac{1}{2}+\frac{s}{2})^2}{\Gamma(1+\frac{s}{2})\pi} & {}_3 F_{2} \left(\frac{-s}{2},\frac{-s}{2},\frac{-s}{2};\frac{1-s}{2},\frac{1}{2};1\right).
		\end{split}
	\end{equation*}
	Hence $H_{2,s}$ can be written as
	\begin{align*}
		H_{2,s}(z) &= \frac{1}{2 \pi} \frac{\tan(\frac{\pi s}{2})}{s+1} z^{1+s} \cdot {}_3 F_{2} \left(\frac{1}{2},\frac{1}{2} ,\frac{1}{2};1 +\frac{s}{2}, \frac{3}{2} + \frac{s}{2} ; \frac{z^2}{16} \right) \\ & \quad+ \frac{\Gamma(s+1)^2}{\Gamma(\frac{s}{2}+1)^4}  \cdot {}_3 F_{2}  \left(-\frac{s}{2} ,-\frac{s}{2},-\frac{s}{2}; \frac{1-s}{2},\frac{1}{2}; \frac{z^2}{16} \right).
	\end{align*}
	It remains to apply Theorem \ref{StellingHW} to arrive at the formula for $W_2(k;s)$.
\end{proof}
\begin{remark}
	Notice that $H_{2,s}(z)$ is not anymore analytic around $z = 0$.
\end{remark}

For odd positive values of $s$ we need to compute the corresponding limits. We present an explicit formula, which can no longer be written in terms of hypergeometric functions.

\begin{thm}
	For odd positive integers $n$ and $|k|<4$,
	\begin{equation}
		\label{W2n}
		W_2(k;n) = (-1)^{\frac{n+1}{2}} \frac{2^n n!}{\pi^3} \cdot G_{3,3}^{2,3} \left( 1 + \frac{n}{2}, 1 + \frac{n}{2},1 + \frac{n}{2};0,\frac{n+1}{2},\frac{1}{2};\frac{k^2}{16}\right).
	\end{equation}
\end{thm}
\begin{proof}
	For $\operatorname{Re}(s)>-1$, not an odd integer, write
	\begin{align*}
		\cos \left( \frac{\pi s}{2} \right) W_2(k;s) = \frac{1}{2 \pi} \frac{\sin \left( \frac{\pi s}{2} \right)}{s+1} F_1 + \frac{\Gamma(s+1)^2 \pi}{\Gamma(\frac{s}{2} + 1)^4 \Gamma(\frac{1+s}{2})} \frac{F_2}{\Gamma(\frac{1-s}{2})},
	\end{align*}
	where $$F_1 = |k|^{1+s} \cdot {}_3 F_{2} \left(\frac{1}{2},\frac{1}{2} ,\frac{1}{2};1 +\frac{s}{2}, \frac{3}{2} + \frac{s}{2} ; \frac{k^2}{16} \right)$$ and $$F_2 = {}_3 F_{2}  \left(-\frac{s}{2} ,-\frac{s}{2},-\frac{s}{2}; \frac{1-s}{2},\frac{1}{2}; \frac{k^2}{16} \right).$$
	Using \emph{Mathematica},
	we have
	\begin{align*}
		\cos \left( \frac{\pi s}{2} \right) & G = 
		-\frac{2^{-1-s} \pi^2}{\Gamma(2+s)} F_1 + \sqrt{\pi} \Gamma \left(-\frac{s}{2} \right)^3 \frac{F_2}{\Gamma(\frac{1-s}{2})},
	\end{align*}
	where $$G = G_{3,3}^{2,3} \left( 1 + \frac{s}{2}, 1 + \frac{s}{2},1 + \frac{s}{2};0,\frac{s+1}{2},\frac{1}{2};\frac{k^2}{16}\right).$$
	Thus, 
	\begin{align*}
		W_2(k;s) &= \left(\frac{\tan \left(\frac{\pi s}{2} \right)}{2 \pi (s+1)} + \frac{2^{-1-s} \pi^{5/2} \Gamma(s+1)^2}{\cos \left(\frac{\pi s}{2} \right) \Gamma(2+s)\Gamma(\frac{s}{2}+1)^4 \Gamma(\frac{1+s}{2}) \Gamma(-\frac{s}{2})^3}\right) F_1\\  & \quad + \left( \frac{\Gamma(s+1)^2 \sqrt{\pi}}{\Gamma \left(\frac{s}{2} + 1\right)^4 \Gamma \left(\frac{1+s}{2} \right) \Gamma \left(-\frac{s}{2}\right)^3}\right) G.
	\end{align*}
	Taking the limit $s \to n$, for $n$ odd and positive, gives the result.
\end{proof}

Using the explicit expression for $W_2(k;s)$ in Theorem \ref{1.6}, the Mahler measure of the Laurent polynomial $k + (x+x^{-1})(y+y^{-1})$ can be computed for $|k|<4$.
\begin{cor}[{\cite[Theorem 3.1]{MR2836402}}]
	For $|k| < 4$,
	\begin{equation}
		\label{MM2}
		\operatorname{m} (k + (x+x^{-1})(y+y^{-1})) = \frac{|k|}{4} {}_3 F_{2} \left(\frac{1}{2},\frac{1}{2} ,\frac{1}{2};1, \frac{3}{2} ; \frac{k^2}{16} \right).
	\end{equation}
\end{cor}
\begin{proof}
	The Mahler measure of $k + (x+x^{-1})(y+y^{-1})$ can be recovered as $$\frac{\d}{\d s} W_2(k;s) |_{s = 0}.$$ Note that in the neighborhood of $s=0$,
	$$\frac{\Gamma(s+1)^2}{\Gamma(\frac{s}{2}+1)^4} = 1 + \mathcal{O}(s^2),$$
	$${}_3 F_{2}  \left(\frac{-s}{2} ,\frac{-s}{2},\frac{-s}{2}; \frac{1-s}{2},\frac{1}{2}; \frac{k^2}{16} \right) = 1 + \mathcal{O}(s^2)$$
	and
	$$\frac{\tan(\frac{\pi s}{2})}{s+1} = \frac{\pi s}{2}+\mathcal{O}(s^2).$$
	Thus, \begin{align*}\frac{\d W_{2}(k;s)}{\d s}\Bigr|_{s = 0} &= \frac{|k|}{4} {}_3 F_{2} \left(\frac{1}{2},\frac{1}{2} ,\frac{1}{2};1, \frac{3}{2} ; \frac{k^2}{16} \right). \qedhere \end{align*}
\end{proof}
\begin{remark}
	The Mahler measure $\operatorname{m} (k + (X+X^{-1})(Y+Y^{-1}))$ can also be written as the double integral
	\begin{equation}
		\label{2varint}
		\operatorname{m} (k + (x+x^{-1})(y+y^{-1})) =  \frac{|k|}{8 \pi}\int_{[0,1]^2} \frac{\d x_1 \, \d x_2}{\sqrt{x_1 x_2(1-x_2)(1-x_1x_2\frac{k^2}{16})}}
	\end{equation}
	using \cite[Eqn. (16.5.2)]{article}.
\end{remark}
For the case $r \geq 3$, we expect that $W_r(k;s)$ cannot be written anymore as a linear combination of hypergeometric functions. 
We now give the proof of Theorem \ref{1.7}.
\begin{proof}[Proof of Theorem \ref{1.7}]
	
	For $|k| < 8$ and real $s>1$,
	we have
	\begin{align}
		\label{W3W3}
		W_3(k;s) &=
		|k|^s\operatorname{Re} \left( {}_{4} F_{3} \left( \frac{-s}{2}, \frac{1-s}{2}, \frac{1}{2}, \dots , \frac{1}{2}; 1, \dots , 1; 
		\frac{64}{k^2} \right)\right) \nonumber \\ &\qquad + \tan \left(\frac{
			\pi s}{2} \right)|k|^s  \operatorname{Im} \left( {}_{4} F_{3} \left( \frac{-s}{2}, \frac{1-s}{2}, \frac{1}{2}, \dots , \frac{1}{2}; 1, \dots , 1; 
		\frac{64}{k^2} \right)\right).
	\end{align}
	Let $\epsilon > 0$. Then for $|k|<8$ we have
	\begin{align*}
		|k|^s {}_{4} F_{3} \left( \frac{-s}{2}, \frac{1-s}{2}, \frac{1}{2},\frac{1}{2} + \epsilon; 1, 1 , 1; 
		\frac{4^r}{k^2} \right) = \alpha_1(s)Y_1(k) + \dots + \alpha_4(s)Y_4(k),
	\end{align*}
	where
	\begin{align*}
		Y_1(k) &= {}_4 F_{3} \left(\frac{-s}{2},\frac{-s}{2},\frac{-s}{2},\frac{-s}{2};\frac{1-s}{2}-\epsilon,\frac{1-s}{2},\frac{1}{2};\frac{k^2}{64}\right),\\
		Y_2(k) &= |k| \cdot {}_4 F_{3} \left(\frac{1-s}{2},\frac{1-s}{2},\frac{1-s}{2},\frac{1-s}{2};\frac{3}{2},1 - \frac{s}{2},1 - \epsilon - \frac{s}{2};\frac{k^2}{64}\right),\\
		Y_3(k) &= |k|^{1+s} {}_4 F_{3} \left(\frac{1}{2},\frac{1}{2},\frac{1}{2},\frac{1}{2};1-\epsilon,1+\frac{s}{2},\frac{3+s}{2};\frac{k^2}{64}\right),\\
		Y_4(k) &= |k|^{1+s+2 \epsilon} {}_4 F_{3} \left(\frac{1}{2} + \epsilon,\frac{1}{2} + \epsilon,\frac{1}{2} + \epsilon,\frac{1}{2} + \epsilon;1+\epsilon,1+\frac{s}{2} + \epsilon,\frac{3+s}{2}+\epsilon;\frac{k^2}{64}\right)
	\end{align*}
	and
	\begin{align*}
		\alpha_1(s) &= \frac{(8 i)^s \Gamma \left(\frac{s+1}{2}\right) \Gamma \left(\epsilon+\frac{s}{2}+\frac{1}{2}\right)}{\Gamma \left(\epsilon+\frac{1}{2}\right) \Gamma \left(\frac{1-s}{2}\right) \Gamma \left(\frac{s}{2}+1\right)^3}, \\
		\alpha_2(s) &= \frac{i^{s+1} 2^{3 s-2} \Gamma \left(\frac{s}{2}\right) \Gamma \left(\epsilon+\frac{s}{2}\right)}{\Gamma \left(\epsilon+\frac{1}{2}\right) \Gamma \left(-\frac{s}{2}\right) \Gamma \left(\frac{s+1}{2}\right)^3},\\
		\alpha_3(s) &=-\frac{i \Gamma (\epsilon) \Gamma \left(-\frac{s}{2}-\frac{1}{2}\right)}{8 \pi ^{3/2} \Gamma \left(\epsilon+\frac{1}{2}\right) \Gamma \left(\frac{1-s}{2}\right)},\\
		\alpha_4(s) &= -\frac{i (-1)^{-\epsilon} 8^{-2\epsilon-1} \Gamma (-\epsilon) \Gamma \left(-\epsilon-\frac{s}{2}-\frac{1}{2}\right) \Gamma \left(-\epsilon-\frac{s}{2}\right)}{\sqrt{\pi } \Gamma \left(\frac{1}{2}-\epsilon\right)^3 \Gamma \left(\frac{1-s}{2}\right) \Gamma \left(-\frac{s}{2}\right)}.
	\end{align*}

	We expand $\alpha_3(s)$ and $\alpha_4(s)$ in powers of $\epsilon$.
	Note that $\Gamma(\epsilon) = \Gamma(\epsilon+1)/ \epsilon$,
	so that for $s$ fixed
	\begin{equation*}
		\alpha_3(s) = \frac{1}{\epsilon} \cdot \frac{i}{4 \pi^2 (1+s)} + \mathcal{O}(1)    
	\end{equation*}
	and
	\begin{equation*}
		\alpha_4(s) = \frac{1}{\epsilon} \cdot \frac{-i}{4 \pi^2 (1+s)} + \mathcal{O}(1),        
	\end{equation*}
	hence
	\begin{align*}
		\lim_{\epsilon \to 0} &\left( \alpha_3(s) Y_3(k) + \alpha_4(s) Y_4(k) \right) =\lim_{\epsilon \to 0}  (\alpha_3(s) + \alpha_4(s)) Y_3(k) - \lim_{\epsilon \to 0} \alpha_4(s) \epsilon \left( \frac{Y_3(k) - Y_4(k)}{\epsilon} \right).
	\end{align*}
	By L'Hôpital's rule,
	\begin{align*}
		\lim_{\epsilon \to 0} & (\alpha_3(s) + \alpha_4(s)) = \frac{\d}{\d \epsilon} \epsilon (\alpha_3(s) + \alpha_4(s)) |_{\epsilon=0}\\
		&=  -\frac{1}{4 \pi (s+1)} + i\frac{1}{4 \pi^2(s+1)} \left(\psi(1) + \psi \left( \frac{-s}{2}\right)+\psi \left( \frac{-s-1}{2} \right) - 3 \psi \left(\frac{1}{2} \right) + \log(256) \right),
	\end{align*}
	where $\psi$ is the digamma function.
	Furthermore, notice that 
	\begin{align*}
		&G_{4,4}^{2,4} \left( \frac{2+s}{2}, \frac{2+s}{2}, \frac{2+s}{2}, \frac{2+s}{2};\frac{1+s}{2}, \frac{1+s}{2}, 0 ,\frac{1}{2} ;\frac{k^2}{64} \right) \\
		&= \frac{\d}{\d \epsilon} \left(- \frac{\Gamma(1-\epsilon) \Gamma(\frac{1}{2} + \epsilon)^4}{8^{1+s+2 \epsilon}\Gamma(\frac{3+s}{2} + \epsilon) \Gamma(1 + \frac{s}{2} + \epsilon)} Y_4(k)+\frac{\Gamma(1+\epsilon) \Gamma(\frac{1}{2})^4}{8^{1+s}\Gamma(\frac{3+s}{2}) \Gamma(1+\frac{s}{2})} Y_3(k)\right) \Biggr|_{\epsilon = 0}\\
		&= \frac{\pi^2}{8^{1+s}\Gamma(\frac{3+s}{2}) \Gamma(1+\frac{s}{2})} \left( \frac{\d}{\d \epsilon} (Y_3(k) - Y_4(k))\Bigr|_{\epsilon =0} - C Y_3(k) \Bigr|_{\epsilon = 0} \right),
	\end{align*}
	where $$C=-2\psi(1) + 4 \psi \left(\frac{1}{2}\right) - \psi\left(1+\frac{s}{2}\right) - \psi \left(\frac{3+s}{2}\right) - \log(64).$$
	Thus we can write
	\begin{align*}
		\lim_{\epsilon \to 0} &\left( \alpha_3(s) Y_3(k) + \alpha_4(s) Y_4(k) \right) \\&=
		\left(\frac{-1}{4 \pi (s+1)} + i\frac{  \cot(\pi s)}{2 \pi (s+1)} \right) |k|^{1+s} {}_4 F_{3} \left(\frac{1}{2},\frac{1}{2},\frac{1}{2},\frac{1}{2};1,1+\frac{s}{2},\frac{3+s}{2};\frac{k^2}{64}\right)  \\
		& \quad + i \frac{4^{s} \Gamma(1+s)}{\pi^{7/2}} G_{4,4}^{2,4} \left( \frac{2+s}{2}, \frac{2+s}{2}, \frac{2+s}{2}, \frac{2+s}{2};\frac{1+s}{2}, \frac{1+s}{2}, 0 ,\frac{1}{2} ;\frac{k^2}{64} \right).
	\end{align*}
	Finally, using \eqref{W3W3} we arrive at the result.
\end{proof}
We can alternatively represent the Meijer $G$-function in \eqref{W3} as the following triple integral.
\begin{prop}
	For $\operatorname{Re}(s)>-1$ and $0 < |k| < 8$,
	\begin{align*}
		&G^{2,4}_{4,4} \left(\frac{2+s}{2},\frac{2+s}{2},\frac{2+s}{2},\frac{2+s}{2};\frac{1+s}{2}, \frac{1+s}{2},0,\frac{1}{2}; \frac{k^2}{64}\right) \\ &\quad = \frac{ \sqrt{\pi}|k|^{1+s}}{\Gamma(1 +s) 2^{3+2s}}\int_{[0,1]^3}  \frac{ (1-x_2)^{\frac{s}{2}} (1-x_3)^{\frac{s-1}{2}} \,  \d x_1 \, \d x_2 \, \d x_3}{\sqrt{x_1 x_2 x_3 (1-x_1)(1-x_1 + \frac{k^2}{64}x_1x_2x_3)}}.
	\end{align*}
\end{prop}
\begin{proof}
	First, for all $z \in \mathbb{C}$ we have
	\begin{equation*}
		G^{2,4}_{4,4} \left(\frac{2+s}{2},\frac{2+s}{2},\frac{2+s}{2},\frac{2+s}{2};\frac{1+s}{2}, \frac{1+s}{2},0,\frac{1}{2}; z \right) = z^{\frac{1+s}{2}} G^{2,4}_{4,4} \left(\frac{1}{2},\frac{1}{2},\frac{1}{2},\frac{1}{2};0, 0,-\frac{1+s}{2},-\frac{s}{2}; z\right).
	\end{equation*}
	Applying Nesterenko's theorem \cite[Proposition 1]{Zud} with $a_0 = \dots = a_3 = \frac{1}{2}$,
	$b_1 = 1$, $b_2 =1 + \frac{s}{2}$ and $b_3 = \frac{3+s}{2}$ to the right hand side gives
	\begin{equation*}
		G^{2,4}_{4,4} \left(\frac{1}{2},\frac{1}{2},\frac{1}{2},\frac{1}{2};0, 0,-\frac{1+s}{2},-\frac{s}{2}; z\right) = \frac{2^s \sqrt{\pi}}{\Gamma(1+s)}\int_{[0,1]^3}  \frac{  (1-x_2)^{\frac{s-1}{2}} (1-x_3)^{\frac{s}{2}} \d x_1 \d x_2 \d x_3}{\sqrt{x_1 x_2 x_3 (1-x_1)(1-x_1 + \frac{k^2}{64}x_1x_2x_3)}}.
	\end{equation*}
\end{proof}

As a consequence of Theorem \ref{1.7} we can compute the Mahler measure of the polynomial $k + (x+x^{-1})(y+y^{-1})(z+z^{-1})$ for $|k|<8$.

\begin{cor}
	The Mahler measure of $k + (x+x^{-1})(y+y^{-1})(z+z^{-1})$ for $|k|<8$ is given by
	\begin{equation}
		\label{MM3}
		\frac{1}{2 \pi^{5/2}} G_{4,4}^{2,4} \left( 1, 1, 1, 1;\frac{1}{2}, \frac{1}{2}, 0 ,\frac{1}{2} ;\frac{k^2}{64} \right).
	\end{equation}
\end{cor}
\begin{proof}
	We expand equation $\eqref{W3}$ in $s$. 
	Note that $$\frac{\tan(\frac{\pi s}{2})^2}{4 \pi (1+s)} = \mathcal{O}(s^2),$$
	\begin{align*}
		{}_4 F_{3} \left(\frac{-s}{2},\frac{-s}{2},\frac{-s}{2},\frac{-s}{2};\frac{1-s}{2},\frac{1-s}{2},\frac{1}{2};\frac{k^2}{64}\right) = \mathcal{O}(s^2),
	\end{align*}
	and
	\begin{align*}
		\frac{4^s \tan(\frac{\pi s}{2}) \Gamma(s+1)}{\pi^{7/2}}&=\frac{1}{2 \pi^{5/2}}s+\mathcal{O}(s^2).
	\end{align*}
	Hence, 
	\begin{align*}
		\frac{\d}{\d s}W_3(k;s) |_{s = 0} &= \frac{1}{2 \pi^{5/2}} G_{4,4}^{2,4} \left( 1, 1, 1, 1;\frac{1}{2}, \frac{1}{2}, 0 ,\frac{1}{2} ;\frac{k^2}{64} \right).  \qedhere
	\end{align*}
\end{proof}
We can further write this Mahler measure for $8>|k|>0$ as the triple integral (compare with equation \eqref{2varint}):
\begin{align}
	\label{3varint}
	&\operatorname{m}(k + (x+x^{-1})(y+y^{-1})(z+z^{-1})) \nonumber \\ & \qquad =\frac{|k|}{16 \pi^2} \int_{[0,1]^3} \frac{\d x_1 \d x_2 \d x_3}{\sqrt{x_1 x_2 x_3(1-x_1)(1-x_2)(1-x_1 + \frac{k^2}{64} x_1x_2x_3)}}.
\end{align}

\section{Concluding Remarks}

Using Theorem \ref{1.5}, explicit formulas for $W_r$ in terms of hypergeometric functions and Meijer $G$-functions can be found for general $r$; see Theorems \ref{1.6} and \ref{1.7}. In this way, the Mahler measure of the polynomials $k + (x_1+x_1^{-1}) \cdots (x_r + x_r^{-1})$ could be written in terms of certain Meijer $G$-functions, although it is not expected that this Mahler measure can be written as a single Meijer $G$-function, like in formulae \eqref{MM2} and \eqref{MM3}.

The result in Theorem \ref{1.3} about the location of the zeros of $W_1$ does not seem to be generalizable to $W_r$. For example, the numerics suggests there is no functional equation in the general case. For the zeros, there seems to be a pattern, though hardly recognizable.

From an arithmetic point of view, it could be interesting to simplify the expression \eqref{W2n} of $W_2(k;n)$ for specific odd integers $n$ and integers $k$. This can be compared to the case $r=1$, where it can be shown, for example, that $W_1(1;n) \in \mathbb{Q} + \frac{\sqrt{3}}{\pi} \mathbb{Q}$.

For integers $s$, the expression \eqref{1.4.i} of $W_r(k;s)$ for $k > 2^r$ is a polynomial in $k$. For the case $r=1$, the induced polynomials are, up to an appropriate transformation of the variable, equal to the Legendre polynomials. This implies that the roots of the induced polynomials have a very specific structure. This structure seems to generalize to the induced polynomials $W_r(k;s)$ for general $r$. More specifically, all the roots of these polynomials seem to lie on the imaginary axis. As these observations seem off-topic we do not pursue this line here.

\,

\emph{Acknowledgements.} The author would like to thank 
Tom Koornwinder for his help with the proof of Theorem \ref{1.3}, and Frits Beukers for his help with the proof of Theorem \ref{1.4}. Thanks to Riccardo Pengo for his helpful comments. Thanks to Wadim Zudilin for his support and useful comments.

\bibliographystyle{amsplain}
\bibliography{main}

\end{document}